\pdfoutput=1
\documentclass[11pt,a4paper]{amsart}
\usepackage{amssymb}
\usepackage{amsthm}
\usepackage{amsmath}
\usepackage{amsxtra}
\usepackage{latexsym}
\usepackage{mathrsfs}
\usepackage[all,cmtip]{xy}
\usepackage[all]{xy}
\usepackage{enumitem}
\usepackage{xcolor}
\usepackage{graphicx}
\usepackage{comment}
\usepackage{mathabx,epsfig}
\usepackage{stmaryrd}
\usepackage{comment}
\usepackage{pst-node}

\usepackage{tikz-cd} 

\usepackage{hyperref}

\newtheorem{thm}{Theorem}[section]
\newtheorem{lem}[thm]{Lemma}

\newtheorem{prop}[thm]{Proposition}

\theoremstyle{definition}
\newtheorem{defn}[thm]{Definition}
\newtheorem{rmk}[thm]{Remark}

\numberwithin{equation}{section}

\newcommand\be{\begin{equation}}
\newcommand\ba{\begin{eqnarray}}
\newcommand\ee{\end{equation}}
\newcommand\ea{\end{eqnarray}}

\def\C{{\mathbb C}}
\def\Q{{\mathbb Q}}

\def\Z{{\mathbb Z}}
\def\P{{\mathbb P}}

\DeclareMathOperator{\Per}{Per}

\DeclareMathOperator{\PGL}{PGL}

\DeclareMathOperator{\ord}{ord}

  {\end{list}}

\title[]
{Preimages Question for Surjective Endomorphisms on $(\P^1)^n$}

\author{Xiao Zhong}
\address{University of Waterloo \\
Department of Pure Mathematics \\
Waterloo, Ontario \\
Canada  N2L 3G1}
\thanks{The author was supported in part by NSERC grant RGPIN-2022-02951}
\email{x48zhong@uwaterloo.ca}
\begin{document}
\begin{abstract}
    Let $K$ be a number field and let $f : (\P^1)^n \to (\P^1)^n$ be a dominant endomorphism defined over $K$.  
    We show that if $V$ is an $f$-invariant subvariety (that is, $f(V)=V$) then there is a positive integer $s_0$ such that 
     $ (f^{-s-1}(V)\setminus f^{-s}(V))(K) = \emptyset$
    for every integer $s \geq s_0$, answering the Preimages Question of Matsuzawa, Meng, Shibata, and Zhang in the case of $(\P^1)^n$.
\end{abstract}
\maketitle
\section{introduction}
Let $X$ be a projective variety and let $f: X \to X$ be a surjective self-maps such that both $X$ and $f$ are defined over a number field $K$. To study the dynamics of $(X, f)$, it is important to identify the closed subvarieties of $Y \subseteq X$ that are invariant under $f$; i.e., subvarieties with $f(Y) \subseteq Y$. For an invariant subvariety $Y$ for the map $f$, it is natural to study its preimages under iterates of $f$. An important principle within arithmetic dynamics is that the underlying geometric structure should exert significant influence on the arithmetic structure, which gives rise to the expectation that the tower of $K$-points: 
$$Y(K) \subseteq (f^{-1}(Y))(K) \subseteq (f^{-2}(Y))(K) \subseteq \cdots $$
should eventually stabilize. This expectation has been made precise in the form of the Preimages Question of Matsuzawa, Meng, Shibata, and Zhang \cite[Question 8.4(1)]{Preimages}.

There are recent works dealing with some special cases of the Preimages Question. Notably, it has been solved with affirmative answer when $X$ is a smooth variety with non-negative Kodaira dimension and $f$ is \'etale \cite[Theorem 1.2]{DynamicalC}. Additionally, if we consider the case when $f = (g,g): X \times X\to X\times X$ is a diagonal map, with $g: X \to X$ is a surjective morphism, then the diagonal subvariety $\Delta \subseteq X \times X$ is invariant subvariety under $f$, and in this special case the Preimages Question becomes a cancellation problem, which asks whether there exists a natural number $s$ with the property that, for all $x, y \in X(K)$, if $g^n(x) = g^n(y)$ for some natural number $n$ then we must in fact have $g^s(x) = g^s(y)$. This special form is a dynamical cancellation problem considered in the work of Bell, Matsuzawa, and Satriano \cite{DynamicalC}, and this question is again answered affirmatively when $X$ is a curve \cite[Theorem 1.3]{DynamicalC} by applying $p$-adic uniformization techniques of Rievera-Letelier \cite{Uniformization}. Considerably more is known for polynomial maps on $\P^1$, and a more general dynamical cancellation results allowing multiple polynomial self-maps on $\P^1$ is also proved along these lines in \cite[Theorem 1.7]{DynamicalC} and \cite[Theorem 1.2]{dynamicalCanPoly}.

Our main result is to give a positive solution to the Preimges Question for surjective endomorphisms on $(\P^1)^n$. It is well-known that a surjective endomorphism of $(\P^1)^n$ has some iterate that becomes a split rational map $(f_1,\ldots, f_n)$.  Since answering the Preimages question for an iterate of a self-map gives one the answer for the original map, our main theorem, stated below, solves the Preimges Question for surjective endomorphisms on $(\P^1)^n$.
\begin{thm} \label{thm: mainfull}
    Let $K$ be a number field, let $n\ge 1$, and let $f = (f_1, \dots, f_n) : (\P^1_K)^n \to (\P^1_K)^n$ be a split rational map defined over a number field $K$ with at least one $f_i$ of degree greater than $1$.  If $V \subseteq (\P^1)^n$ is a subvariety defined over $K$ that is invariant under $f$ then there exists a non-negative integer $s_0$ such that 
     $$ (f^{-s-1}(V) \setminus f^{-s}(V))(K)  = \emptyset$$
     for all $s \geq s_0$.
\end{thm}
It is also well known that Theorem \ref{thm: mainfull} holds when $n=1$.
\begin{rmk}
    This paper works over a number field $K$ for the reader's convenience, but the argument also works if we let $K$ instead be a finitely generated extension of $\Q$. Proposition \ref{prop : toris} uses the fact that the set of roots of unity in the union of all finite field extensions of $\Q$ of bounded degree is finite and it remains true that the set of roots of unity inside the union of field extensions of bounded degree over $K$ is finite when $K$ is a finitely generated extension over $\Q$. Proposition \ref{prop: Lattes} uses \cite[Theorem 2.3]{DynamicalC} whose proof is based on embedding $K$ into a finite extension of $\Q_p$ for a suitable prime $p$, which works for any finitely generated field extension of $\Q$ as well (see \cite[Proposition 2.5.3.1]{DML-Book}). Everything else goes through directly without additional changes when considering $K$ as a finitely generated field extension of $\Q$. 
\end{rmk}
\section{Proof of the Main Theorem}
In this section we give the proof of the main result.  Our first result shows that we can reduce to the case when our maps all have degree $\ge 2$.
\begin{lem}\label{lem:reduce-to-alldegreenontrivial}
    Let $n\ge 2$ be a natural number and let $f = (f_1, \dots, f_n) : (\P^1)^n \to (\P^1)^n$ be a split rational map defined over a algebraically closed characteristic zero field and suppose that there exists a positive integer $k \in \{1, \dots ,n-1\}$ such that $\deg(f_i) > 1$ when $1 \leq i \leq k$ and $\deg(f_i) = 1$ when $i > k$. If $V$ is an irreducible subvariety of $(\P^1)^n$ that is invariant under $f$ then there exist subvarieties $V_1 \subseteq (\P^1)^k$ and $V_2 \subseteq (\P^1)^{n-k}$ such that $V=V_1\times V_2$ and such that  $V_1$ is invariant under $(f_1, \dots, f_k)$ and $V_2$ is invariant under $(f_{k+1}, \dots, f_n)$.
\end{lem}
\begin{proof}
    The proof is similar to \cite[Proposition 3.14]{Xie2023}. Let $d$ denote the dimension of $V$ and let $\pi_1 : (\P^1)^n \to (\P^1)^k$ and $\pi_2 : (\P^1)^n \to (\P^1)^{n-k}$ be respectively the projection onto the first $k$ factors and the projection onto the last $n-k$ factors. 
    
    Let $V_i = \pi_i(V)$ and $d_i = \dim(V_i)$ for $i=1,2$. Since $V\subseteq V_1\times V_2$ we have $d \leq d_1 + d_2$ and $d = d_1 + d_2$ if and only if $V = V_1 \times V_2$. 
    
 Thus we may assume without loss of generality that $d < d_1 + d_2$. We take $\alpha_1$ to be numerical class of the ample line bundle $\pi_1^*\mathcal{O}_{(\P^1)^k}(1,1, \dots,1)$ and $\alpha_2$ to be the numerical class of the ample line bundle $\pi_2^*\mathcal{O}_{(\P^1)^{n-k}}(1,1, \dots,1)$.  Furthermore, we let $\beta_j$ denote the numerical class of the line bundle $p^*_{j}\mathcal{O}_{\P^1}(1)$ for each $j \in \{1, 2, \dots, n\}$ and $p_j : (\P^1)^n \to \P^1$ the projection onto the $j$-th coordinate. Then 
    \begin{equation}
        \alpha_1 = \beta_1 + \cdots + \beta_k,
    \end{equation}
    \begin{equation}
        \alpha_2 = \beta_{k+1} + \cdots + \beta_n.
    \end{equation}
    Notice that for any $j \in \{0, 1, \dots, d\}$, $\alpha^j_1 \cdot \alpha^{d-j}_2 \cdot V \geq 0$ and it is positive if $j = d_1$ or $d-j = d_2$. Let's denote $I_1 = \{1,2, \dots, k\}$ and $I_2 = \{k+1, \dots, n\}$ from now on.

    Since there are only finitely many collections of indices (allowing repetition) of size $d-d_2$ which is inside $I_1$ and also only finitely many collections of indices (allowing repetition) of size $d_2$ which is inside $I_2$, there exists a $\{r_1, r_2, \dots, r_{d-d_2}\} \subseteq I_1$, and a $\{e_1, e_2, \dots, e_{d_2}\} \subseteq I_2$ such that 
    $$ V \cdot \prod^{d-d_2}_{t= 1} \beta_{r_t} \prod^{d_2}_{l = 1} \beta_{e_l} > 0,$$
    and $C = \prod^{d-d_2}_{t= 1} c_{r_t}$, where $c_i = \deg(f_i)$ for $i \in I$, is the maximum in the set 
    $$ \left\{ \prod_{i \in I'} c_{i} :V \cdot \prod_{i \in I'} \beta_i \prod_{v \in J'} \beta_v >0, I' \subseteq I_1, J' \subseteq I_2, |I'| = d -d_2, |J'| = d_2 \right\}.$$
    
    We let $V'$ be an irreducible subvariety of $ V \cap \bigcap^{d-d_2}_{t = 1}\beta_{r_t}$ of dimension $d_2$ and here we abuse notation to view $\beta_{r_t}$'s as some suitable hypersurfaces in the numerical classes $p_{r_t}^*\mathcal{O}_{\P^1}(1)$'s. Then denote $d'_1 = \dim(\pi_1(V'))$. We have 
   $$ V' \cdot (\alpha_1)^{d'_1} \cdot (\alpha_2)^{d_2 - d'_1} > 0.$$
   This implies that there exists collections of indices 
   $$ \{j_1 = r_1, j_2 = r_2, \dots, j_{d-d_2} = r_{d-d_2}, j_{d -d_2 + 1}, \dots, j_{d + d'_1 - d_2}\} \subseteq I_1$$
   and 
   $$ \{u_1 ,u_2, \dots ,u_{d_2-d'_1}\} \subseteq I_2$$
   such that
  \begin{equation}\label{eq: larger-index} 
   V \cdot \prod^{d+ d'_1 - d_2}_{ t= 1} \beta_{j_t} \prod^{d_2 - d'_1}_{l = 1} \beta_{u_l} >  0.
  \end{equation}
   Notice that  
   \begin{equation} \label{eq: larger-coefficients}
       \prod^{d+ d'_1 - d_2}_{t = 1} c_{j_t} > C
   \end{equation} 
   by construction and Equation (\ref{eq: larger-index}) implies that
   $$ V \cdot \alpha^{d + d'_1 - d_2}_1 \cdot \alpha^{d_2 - d'_1}_2 > 0.$$

    Let $u_1, u_2 \in \mathbb{R}$, we have 
    \[
        (\deg(f|_V))(V \cdot (u_1\alpha_1 + u_2 \alpha_2)^d) = f_*(V)\cdot ((u_1\alpha_1 + u_2 \alpha_2)^d))
    \]
    \[  
        = V \cdot (u_1 (g_1)^*(\alpha_1) + u_2 (g_2)^*(\alpha_2))^d
    \]
    \[ 
     = V \cdot \left(u_1 \sum^k_{i = 1} c_i \beta_i + u_2 \sum^n_{i = k+1} c_i \beta_i\right)^d.
    \]

    Now we compare the coefficients of the $u^{d + d'_1 - d_2}_1u^{d_2-d'_1}_2$ and $u^{d -d_2}_1u^{d_2}_2$ terms and we obtain that for each positive integer $m$:
    \begin{align*}
        \deg(f^m|_V) \alpha_1^{d + d'_1 - d_2} \cdot \alpha_2^{d_2-d'_1} \cdot V 
      = \left(\sum^k_{i = 1} c^m_i \beta_i\right)^{d + d'_1 - d_2} \cdot \left(\sum^n_{i = k+1} c^m_i \beta_i\right)^{d_2-d'_1} \cdot V, 
    \end{align*}
    and
     \begin{align*}
         \deg(f^m|_V) \alpha_1^{d - d_2} \cdot \alpha_2^{d_2} \cdot V = \left(\sum^k_{i = 1} c^m_i \beta_i\right)^{d- d_2} \cdot \left(\sum^n_{i = k+1} c^m_i \beta_i\right)^{d_2} \cdot V.
    \end{align*}
    Therefore 
    \begin{equation}\label{eq: degf1}
         \deg(f^m|_V) = \left(\sum^k_{i = 1} c^m_i \beta_i\right)^{d + d'_1 - d_2} \cdot \alpha_2^{d_2-d'_1} \cdot V/\left(\alpha_1^{d + d'_1 - d_2} \cdot \alpha_2^{d_2-d'_1} \cdot V\right)
    \end{equation}
       and
   \begin{equation}\label{eq: degf2}
  \deg(f^m|_V)       =  \left(\sum^k_{i = 1} c^m_i \beta_i\right)^{d- d_2} \cdot \alpha_2^{d_2} \cdot V/\left(\alpha_1^{d - d_2} \cdot \alpha_2^{d_2} \cdot V\right),
   \end{equation}
since $c_i = 1$ when $i \in I_2$, 
    \[\alpha_1^{d + d'_1 - d_2} \cdot \alpha_2^{d_2- d'_1} \cdot V > 0,\]
    and
    \[\alpha_1^{d - d_2} \cdot \alpha_2^{d_2} \cdot V > 0.\]
   
   Now, by (\ref{eq: larger-coefficients}) and Equation (\ref{eq: degf1}), we have 
   $$\deg(f^m|_V) = a M^m + o(M^m) $$
   for some positive integer $a$ and $M$, such that $M > C \geq  \prod^{d - d_2}_{t = 1} c_{l_t}$ for all $\{l_1, \dots, l_{d-d_2}\} \subseteq I_1$ and $\{v_1, \dots, v_{d_2}\} \subseteq I_2$ with the property
   \[ \prod^{d -d_2}_{t=  1} \beta_{l_t} \prod^{d_2}_{s = 1}\beta_{v_s} \cdot V > 0.\]
   
   On the other hand, Equation (\ref{eq: degf2}) gives that $\deg(f^m|_V) = \mathcal{O}(C^m)$. This is a contradiction and so the result follows.
   
\end{proof}
Lemma \ref{lem:reduce-to-alldegreenontrivial} allows us to reduce to the case that all maps $f_i$ in the statement of Theorem \ref{thm: mainfull} have degree at least two when proving the result, and we now consider this restricted case.  We begin with a result in which the maps are special.
\begin{defn}
    We define the Chebyshev polynomial of degree $r$ to be the unique polynomial $T_r$ of degree $r$ such that 
    $$ T_r((x+ x^{-1})/2) =(x^r + x^{-r}) / 2 $$
    for any $x \in \mathbb{C}^*$.
\end{defn}
\begin{prop} \label{prop : toris}
     Let $K$ be a number field, let $n$ and $D$ be positive integers that are at least $2$, and let $K'$ be the union of all finite field extensions of $K$ of degree not greater than $D$ in some fixed algebraic closure of $K$.  Suppose that $f = (f_1, \dots, f_n) : (\P^1_K)^n \to (\P^1_K)^n$ is a split rational map defined over $K$ in which the $f_i$'s all have the same degree $d \ge 2$ and each $f_i$ is conjugate to either $x^{\pm d}$ or $\pm T_{d}$ for $i \in \{1, 2, \dots ,n\}$.  If $V \subset (\P^1)^n$ is an irreducible $f$-invariant hypersurface defined over $K$, then exists a non-negative integer $s_0$ such that 
     $$ (f^{-s-1}(V) \setminus f^{-s}(V))(K')  = \emptyset$$
     for all $s \geq s_0$.
\end{prop}
\begin{proof}
     We first enlarge our number field $K$ to a larger number field so that each $f_i$ can be conjugated to either $x^{\pm d}$ or to $\pm T_{d}$ by a map $v_i \in \PGL_2(K)$ and $\{i, e^{\pi i /4}\} \subseteq K$. Notice that it is enough to prove the statement in this larger field we henceforth assume that $K$ is a number field satisfying this condition.  We can then perform a change of coordinates and may assume without the loss of the generality that $f$ is of the form $f = (f_1(x_1), \dots, f_n(x_n)) = (x_1^{\pm d}, x_2^{\pm d}, \dots, x_k^{\pm d}, \pm T_{d}(x_{k+1}), \dots, \pm T_{d}(x_n) )$ for some $k\in \{0,\ldots ,n\}$. 
     
     Let $$\mu =  (x_1, \dots , x_k,  (x_{k+1} + x^{-1}_{k+1})/2, \dots, ( x_n +  x^{-1}_n)/2) : (\C^*)^n \to (\P^1)^n,$$
 and 
$$ G'  = (x_1^{\pm d}, \dots, x_k^{\pm d},  \pm x_{k+1}^d, \dots, \pm x^d_n) : (\C^*)^n \to (\C^*)^n$$
such that $\mu \circ G' = f \circ \mu$.
Then take 
$$ \mu' = (x_1, \dots, x_k, \xi_{k+1} x_{k+1}, \dots, \xi_n x_n)$$
and 
$$ G = (x^{\pm d}_1, \dots, x^{\pm d}_k, x^d_{k+1}, \dots, x^d_n),$$
where $\xi_{k+1}, \dots, \xi_n$ are some roots of unity such that 
$\mu' \circ G = G' \circ \mu'$.
Enlarge $K$ so that $\{\xi_{k+1}, \dots, \xi_n\} \subseteq K$ and abuse notation from now on to let $\mu$ be $\mu \circ \mu'$. So we have 
$$\mu \circ G = f \circ \mu.$$
Notice that further abusing notations by replacing $G$ by $G^2$ and $f$ by $f^2$, we may let

$$ G = (x_1^{d}, \dots, x_k^{ d}, x_{k+1}^d,\dots,   x^d_n)$$
and
$$ \mu \circ G = f \circ \mu. $$

Now we prove the statement by induction on the dimension $n$. If $n = 1$, then the statement is certainly true as $V$ is a just a finite set of points. Now, we suppose it is true when $n \leq n_0$ for some positive integer $n_0$ and we want to show when $n = n_0 + 1$. From now on, let $n = n_0 + 1$.
Let $W $ be the preimage of $V$ under $\mu$, which is a subvariety in $(\C^*)^n$. Since $f(V) = V$, we have 
$$ G(W) \subseteq W.$$ Thus, there exists a subvariety $W' \subseteq W$ such that $G(W') = W'$, after replacing $f$ and $G$ by a suitable iterate, we have $W \subseteq G^{-1}(W')$. Therefore, by \cite[Lemma 10]{invariant-varieties-group}, we have $W'$ is a translation of an algebraic subgroup of $(\C^*)^n$. Notice that $(\xi_ix_i + \xi^{-1}_ix^{-1}_i)/2$ maps $\{0 , \infty\}$ to $\{\infty\}$, for each $i \in \{k+1, \dots, n\}$. Thus, we have $$V' = V \setminus \mu(W) =\left(\bigcup^n_{i= 1} (V \cap H_i)\right)\cup \bigcup^k_{i = 1} (V \cap P_i),$$ where $H_i$ is the hypersurface in $(\P^1)^n$ defined by having $i$-th coordinate equals to $\infty$ and $P_i$ is the hypersurface with the $i$-th coordinate equals to $0$. We write $V' = \bigcup^{k_0}_{j = 1}V_j$ for some positive integer $k_0$, where $V_j$'s are irreducible components of $V \cap H_i$ or $V \cap P_l$ for some $i \in \{1,2, \dots, n\}$ and $l \in \{1,2, \dots, k\}$. Also, $f(V') = V'$ because $f(V) = V$, $f(V') \subseteq V'$ and $f(\mu(W)) \subseteq \mu(W)$ as power maps take $0$ or $\infty$ to themselves and Chebyshev polynomials take $\infty$ to itself. 

After replacing $f$ by some iteration we may assume each $V_j$ is invariant under $f$. Notice that every irreducible component of $V\cap H_i$ or $V \cap P_l$ has dimension at least $n-2$, if not empty, for each $i \in \{1,2, \dots, n\}$ and $l \in \{1,2, \dots, k\}$. Thus, after some reorderings of the coordinates, each $V_j$ projects onto either an irreducible hypersurface in $(\P^1)^{n-1}$ or $(\P^1)^{n-1}$ itself, denoted as $V'_j$ in both case, and projects to either $0$ or $\infty$ in the remaining $\P^1$ factor. If $V'_j = (\P^1)^{n-1}$, then its preimages under $f$ restricted to $(\P^1)^{n-1}$ is itself and there is nothing to show, so from now on we assume $V'_j$ is an irreducible hypersurface. Also, $V'_j$ is invariant under $f_j'$, the restriction of $f$ to the corresponding $(\P^1)^{n-1}$ factors. Now the induction hypothesis tells us that there exists non-negative integer $s'_j$ such that 
 $$ ((f_j')^{-s-1}(V'_j) \setminus (f_j')^{-s}(V'_j))(K')  = \emptyset$$
 for every $s \geq s'_j$. But notice that this also implies 
 $$ (f^{-s-1}(V_j) \setminus f^{-s}(V_j))(K')  = \emptyset$$
 for $s \geq s'_j$, since preimages of $\{\infty, 0\}$ under $x^d$ and preimages of $\infty$ under $\pm T_d$ are just themselves.
 Take $s'_0 = \max^{k_1}_{j = 1}s'_j$.
 
 Now we look at the preimages of $V$ under $\mu$. If this is empty, then $V$ lives completely inside some $H_i$ or $P_j$ for some $i \in \{1,2, \dots, n\}$ and $j \in \{1,2, \dots, k\}$ and, using arguments above, we go to the induction steps. From now on, we assume that $W = \mu^{-1}(V)$ is not empty. Notice that $\mu^{-1}((V \setminus V')(K')) \subseteq W(L)$, where $L$ is the union of finite field extensions of degree $D \cdot 2^n$ over $K$. Recall that $W \subseteq G^{-1}(W')$ and $G(W') = W'$. It is enough to show that there exists a non-negative integer $s_1$ such that for any $x \in (L^*)^n$, if $G^s(x) \in W'(L)$ for some $s \geq 0$, then $G^{s_1}(x) \in W'(L)$, where $L^*$ denotes the set of elements in $L$ that is not zero. As, if this $s_1$ exists and there still exists a $y \in (\P^1_{K'})^n$ and a $s > s_1 $ such that $f^{s }(y) \in \mu(W)(K') $ but $f^{s_1 }(y) \notin \mu(W)$, then for any $x \in \mu^{-1}(y) \subset (L^*)^n$, $$G^{s+1}(x) \in G(\mu^{-1}(f^{s }(x))) \subset G(W(L)) \subseteq W'(L).$$ This implies that $G^{s_1} (x) \in W'(L)$. But we have $G^{s_1}(x) \notin W$ from $f^{s_1}(y) \notin \mu(W)$. This is a contradiction. 
 
 Now, recall $W'$ is a torsion translation of an algebraic subgroup. So we have $W' = V(x^{r_1}_1\dots x^{r_n}_n - \epsilon)$ where $r_1, \dots, r_n$ are integers that are not all zero and $\epsilon$ is an torsion element such that $\epsilon^d = \epsilon$. Enlarge $K$, $K'$ and $L$ if necessary so that $\epsilon \in K$. Then $L$-points in preimages of $W'(L)$ under iterates of $G$ live in the union of $  V(x^{r_1}_1\dots x^{r_n}_n - \epsilon \lambda)$, where $\lambda$ ranges over elements in $L$ which is $d^m$-torsion for some non-negative integer $m$. Notice that if $(x_1, \dots, x_n) \in (L^*)^n$, then $x_i \in L_i$ where $L_i$ is a finite field extension of $K$ such that $[L_i :K] \leq D \cdot 2^n$ for each $i \in \{1,2, \dots, n\}$. Therefore, $\lambda \in L' = L_1L_2\dots L_n$ where $L'$ is a finite field extension of $K$ such that $[L':K] \leq (D \cdot 2^n)^n$. Therefore, $\lambda$ is a root of unity such that 
 $$[\Q(\lambda) : \Q] \leq (D \cdot 2^n)^n \cdot [K : \Q].$$ 
 
 We claim that 
$$ M =  \{ \lambda \in \overline{\Q} : \lambda^{d^m} = 1, \text{ for some } m \in \mathbb{N}^+, [\Q(\lambda) : \Q] \leq (D \cdot 2^n)^n \cdot [K : \Q] \} $$
is a finite set. This is because the set of roots of unity of bounded degree is finite, which is a special case of Northcott property, and $M$ is a subset of it. Then we can take $s_1 = \max_{\lambda \in M} \ord(\lambda)$, where $\ord(\lambda)$ is the minimum non-negative integer $m$ such that $\lambda^{d^m} = 1$. Then for any $y \in (L^*)^n$ such that $G^s(y) \in W'(L)$ for some $s > s_1$, we have $y \in V(x^{r_1}_1\dots x^{r_n}_n - \epsilon \lambda)$ for some $\lambda \in M$ and thus $G^{s_1}(y) \in W'(L)$. This will conclude the proof in this case.  

To conclude, we take $s_0 = \max(s'_0, s_1)$. We have 
 $$ (f^{-s-1}(V) \setminus f^{-s}(V))(K')  = \emptyset$$
 for any $s \geq s_0$.
\end{proof}
\begin{prop}\label{prop: Lattes}
    Let $K$ be a number field, let $n$ and $D$ be positive integers that are at least $2$, and let $K'$ be the union of all finite field extensions of $K$ of degree not greater than $D$ in some fixed algebraic closure of $K$.  Suppose that $f = (f_1, \dots, f_n) : (\P^1_K)^n \to (\P^1_K)^n$ is a split rational map defined over $K$ in which the $f_i$'s all have the same degree $d \ge 2$ and each $f_i$ is a Latt\`es map for $i \in \{1,2, \dots, n\}$.  If $V \subset (\P^1)^n$ is an irreducible $f$-invariant hypersurface defined over $K$ that projects dominantly onto any subset of $n-1$ coordinate axes, then exists a non-negative integer $s_0$ such that 
     $$ (f^{-s-1}(V) \setminus f^{-s}(V))(K')  = \emptyset$$
     for all $s \geq s_0$.
\end{prop}
\begin{proof}
     In this case, $f = (f_1, f_2, \dots, f_n)$ where each $f_i$'s is a Latt\`es map. We consider a diagram:
\begin{equation}
    \begin{tikzcd}
E_1 \times E_2 \times \dots \times E_n \arrow{r}{G} \arrow[swap]{d}{ \pi} & E_1 \times E_2 \times \dots \times E_n \arrow{d}{  \pi} \\
(\P^1)^n \arrow{r}{f} & (\P^1)^n
\end{tikzcd}
\end{equation}
where $G = (g_1, g_2, \dots, g_n)$ is also a split morphism on the abelian variety $E_1 \times E_2 \times \dots \times E_n$ and $\pi = (\pi_1, \pi_2, \dots, \pi_n)$ is a projection map such that each $g_i$ and $\pi_i$ satisfies $f_i \circ \pi_i = \pi_i \circ g_i$ by the fact that $f_i$ is a Latt\`es map. Let $W = \pi^{-1}(V)$, which is a subvariety in $E_1 \times E_2 \times \dots \times E_n$. Since each $\pi_i$ has degree bounded by $6$ \cite[Proposition 6.37]{book-arithmetic-ds} and we can enlarge $K$ to a larger number field so that $g_i$'s are defined over $K$, we replace $K$ by this larger number field. It is enough to show that $L$-points of $W$ are stabilized under preimages of $G$, for every finite field extension $L$ of $K$ of degree not greater than $6^n \cdot D$. We have $G(W) \subseteq W$ by construction and there exists a subvariety $W' \subseteq W$ such that, after replacing $G$ and $f$ by some suitable iterate, $G(W') = W'$ and $G(W) \subseteq W'$. We just need to show that there exists a non-negative integer $s_0$ such that for any $s \geq s_0$ and $[L : K] \leq 6^n\cdot D$,
\begin{equation}\label{eq: reduce-to-irred}
    (G^{-s-1}(W') \setminus G^{-s}(W'))(L) = \emptyset.
\end{equation}
 This will imply that 
\begin{equation}
    (f^{-s-1}(V) \setminus f^{-s}(V))(K) = \emptyset,
\end{equation}
since, if not, there exists $x \in (\P^1_K)^n$ such that $f^s(x) \in V(K)$ but $f^{s_0}(x) \notin V(K)$ for some $s > s_0$. Thus there exists a $$y \in \pi^{-1}(x) \subset (E_1 \times E_2 \times \dots \times E_n)(L),$$ where $L$ is a finite field extension of $K$ of degree not greater than $6^n\cdot D$, such that $G^{s+1}(y) \in W'(L)$, but $G^{s_0}(y) \notin W$ so is not in $W'$. This is a contradiction to (\ref{eq: reduce-to-irred}).

Notice that it is enough to prove \ref{eq: reduce-to-irred} for each irreducible component of $W'$. From now on we abuse notation and let $W$ denote an arbitrary irreducible component of $W'$. To prove the statement (\ref{eq: reduce-to-irred}), we first consider $\Tilde{W}$ which is a translation of $W$ by some $\Bar{K}$-point in $E_1 \times E_2 \times \dots \times E_n$ such that $\Tilde{W}$ is invariant under $\Tilde{G}$, where $\Tilde{G}$ is a group homomorphism of the abelian variety with the property that $G$ is the composition of $\Tilde{G}$ with a suitable translation. 

Then, by \cite[Theorem 3.1]{invariant-varieties-abelian}, $\Tilde{W}$ contains a Zariski dense set of preperiodic points of $\Tilde{G}$.
Note that the preperiodic points of $\Tilde{G}$ are torsion points in $$E_1 \times E_2 \times \dots \times E_n$$ (see \cite[Claim 3.2]{invariant-varieties-abelian} and notice that $\Tilde{G}$, $G$ and $f$ are polarizable). Therefore, we can apply the classical Manin-Mumford conjecture (proved in \cite{classicalMMconjecture}) to get that $\Tilde{W}$ is a translation of some subabelian variety of $E_1 \times E_2 \times \dots \times E_n$ and so is $W$. Thus $G$ restricted to $W$ is an \'etale morphism and therefore we apply \cite[Theorem 2.3]{DynamicalC} to conclude that there exists a nonnegative integer $s_0$ such that for any $s \geq s_0$ we have
\begin{equation}
    (G^{-s-1}(W) \setminus G^{-s}(W))(L) = \emptyset,
\end{equation}
for any finite field extension $L$ of $K$ of degree not greater than $6^n \cdot D$.
\end{proof}
\begin{lem}\label{lem: equal-degree}
    Let $V$ be an irreducible hypersurface in $(\P^1)^n$ that projects dominantly onto every subset of $n-1$ coordinate axes. Let $f = (f_1, \dots, f_n)$ be a split rational map with $\deg(f_i) > 1$ for all $1 \leq i \leq n$.  If $f(V) = V$ then $\deg(f_1) = \deg(f_2) = \dots = \deg(f_n)$.
\end{lem}
\begin{proof}
   First of all, if $n = 2$ we see that the statement is true by \cite[Theorem 4.1]{Invariant-curve-rational}. Now, we assume $n > 2$ and we just need to show that for any $i,j \in \{1, 2, \dots , n\}$ and $i \neq j$, we have $\deg(f_i) = \deg(f_j)$. Without loss of generality, we assume $i = 1$ and $j = 2$. Notice that since $V$ projects dominantly onto every subset of $n-2$ coordinate axes and cooridinate projection maps are closed, if $\pi'$ is projection onto the last $(n-2)$ coordinates, then we have $\pi'(V) = (\P^1)^{n-2}$, as $\pi'(V)$ is Zariski dense and closed in $(\P^1)^{n-2}$. 
   
   Thus $\prod^n_{i = 3} \Per(f_i) \subseteq \pi'(V)$. We take $U = \prod^n_{i = 3} \Per(f_i)\cap (\C^*)^{n-2}$, which is also a Zariski dense set in $(\P^1)^{n-2}$, where ${\rm Per}(h)$ denotes the set of periodic points of a map $h$.
    
   We note that $V \cap (\C^*)^n$ is the zero set of a 
   polynomial $h(x_1, \dots, x_n)$ that for each $i$ has a non-trivial term 
   involving the variable $x_i$. We wish to show that 
   there exists a $\alpha \in U$ such that $h(x_1, x_2, 
   \alpha_3, \dots, \alpha_n)$ cannot be written as 
   $h_1(x_1)h_2(x_2)$ for some polynomials $h_1$, $h_2$. 
   Notice that we can write 
   $$h(x_1, x_2, \dots, x_n) = \sum^m_{k = 0} 
   \gamma_k(x_2) x^k_1,$$ 
   for some positive integer $m$ and coefficients
   $\gamma_k(x_2) \in \C[x_3, \dots, x_n][x_2]$ for each $k$. 
   Notice that the degree of $\gamma_k$ is bounded and we let $l$ be a bound on this degree. 
   
   Then we construct a $k \times l$ 
   matrix $M = (\beta_{i,j})_{0 \leq i \leq k, 0 \leq j \leq l}$, 
   where for each $i,j$, $\beta_{i,j} \in \C[x_3, \dots , x_n]$ is the 
   coefficients of $x_2^j$ in $\gamma_i$. 
   
   Notice that for every $
   (\alpha_3, \dots, \alpha_n) \in (\C^*)^n$, $h(x_1, x_2, 
   \alpha_3, \dots, \alpha_n)$ can be written as 
   $h_1(x_1)h_2(x_2)$ if and only if $M(\alpha_3, \dots ,
   \alpha_n)$ has rank not greater than $1$, which is 
   equivalent to each $2 \times 2$ minor in the matrix $M$ having trivial determinant. So suppose that for every $(\alpha_3, \dots, \alpha_n) \in U$, the $2 \times 2$ minors of $M(\alpha_3, \dots, \alpha_n)$ all have trivial determinant. 
   
   Then the $2 \times 2$ minors of $M(x_3, \dots, x_n)$ are identically zero as they vanish on the Zariski dense set $U$ in $(\P^1)^{n-2}$. This them implies that $h(x_1, x_2, \dots, x_n) = h_1(x_1)h_2(x_2)$ for some $h_1, h_2 \in \C[x_3, \dots , x_n]$, which contradicts the assumption that $V$ is irreducible. 
   
   Therefore there exists $(\alpha_3, \dots, \alpha_n) \in U$ such that $h(x_1, x_2, \alpha_3, \dots, \alpha_n)$ cannot be written as product of two single variable polynomials and thus it has an irreducible component $\Tilde{h}(x_1,x_2)$ such that $C = V(\Tilde{h}(x_1,x_2))$ is an irreducible curve projecting dominantly onto each coordinate in $(\P^1)^2$. 
   
  Notice that there exists a positive integer $d$ such that $f^d_i(\alpha_i) = \alpha_i$ for each $3 \leq i \leq n$. Thus $(f^d_1, f^d_2)(V(h(x_1, x_2, \alpha_3, \dots, \alpha_n))) = V(h(x_1, x_2, \alpha_3, \dots, \alpha_n))$ and if we replace $d$ by some larger integer, we have $(f_1^d, f_2^d)(C) = C$. Now since the statement is known for $n = 2$ we have that $\deg(f_1) = \deg(f_2)$, and so it now follows that $\deg(f_1) = \deg(f_2) = \dots = \deg(f_n)$.
   %
    
    
\end{proof}
\begin{lem} \label{lem: red-to-surface} 
    In the case where the degrees of $f_1,\ldots ,f_n$ are all at least two, is enough to prove Theorem \ref{thm: mainfull} with the additional assumption that $V$ is an irreducible hypersurface of dimension not less than $1$ that projects dominantly onto every subset of $n-1$ coordinate axes in $(\P^1)^n$.
\end{lem}
\begin{proof}
    The idea is the same as in \cite[Proposition 2.1]{invariant-varieties-splitting}. We assume throughout this proof that we have the hypotheses of Theorem \ref{thm: mainfull} and that the maps $f_1,\ldots ,f_n$ all have degree $\ge 2$.

    We use induction on the dimension of $V$. If $\dim(V) = 0$, then $V$ is a finite set of points and Theorem \ref{thm: mainfull} is true since the preimages of $K$-points in $V$ live inside a set of $K$-points of bounded height. Now suppose that the conclusion to Theorem $\ref{thm: mainfull}$ holds when $\dim(V) <D$, we prove the case that $\dim(V) = D$. Notice that it is enough to prove that the conclusion to Theorem \ref{thm: mainfull} holds for each irreducible components of $V$ separately after replacing $f$ by a suitable iterate. 
    
    Thus we assume that $V$ is irreducible. Then there exist $D$ coordinate axes such that the projection $\pi$ of $V$ onto these axes is dominant. 
    
    Without loss of generality, we assume that they are the first $D$ coordinates and for $j>D$ we let $\pi_j$ denote the projection from $V$ to the coordinate axes indexed by $\{1,\ldots , D,j\}$.
    
     Therefore $\pi_j(V)$ is a hypersurface in $(\P^1)^{D+1}$ and also 
     $$H_j := \pi_j(V) \times (\P^1)^{n - D  - 1}$$ is a hypersurface in $(\P^1)^n$. Now we claim that $V $ is a component of $\bigcap^n_{j = D + 1} H_j$. Notice that $\dim(\bigcap^n_{j = D + 1} H_j) \geq D$ and $V \subset \bigcap^n_{j = D + 1} H_j$. So we just need to show that $\dim(\bigcap^n_{j = D + 1} H_j) = D$. 
     
     Since $\pi_j(V)$ projects dominantly onto the first $D$ coordinate axes, there exists a Zariski open subset $U \subset (\P^1)^{D}$ such that for each 
     $$\alpha = (\alpha_1, \dots , \alpha_D) \in U$$ there exists a finite set $S_{\alpha, j}$ such that if $(\alpha_1, \dots, \alpha_n) \in H_j$ and $(\alpha_1, \dots, \alpha_{D}) \in \pi(V)$ then $\alpha_j \in S_{\alpha, j}$. This is saying for each $\alpha \in U$ there are only finitely many points in $\bigcap^n_{j = D + 1} H_j$ such that the first $D$ coordinates are equal to $\alpha$. Therefore, $\dim(\bigcap^n_{j = D + 1} H_j) = D$ and $V$ is a component of $\bigcap^n_{j = D + 1} H_j$. Notice that since $f(V) = V$, we also have $f(H_j) = H_j$.
    
    We claim that if for each $H_j$, we have the statement of Theorem \ref{thm: mainfull} holds, then certainly it holds for $V$. To prove the claim, we first notice that there exists a subvarity $H \subseteq \bigcap^n_{j = D + 1} H_j$ containing $V$ such that $f(H) = H$ and $f(\bigcap^n_{j = D + 1} H_j) = H$ if we replace $f$ by some iterate. We denote $V_1, \dots, V_k $ as the irreducible components of $H$ of dimension $D$ and without loss of generality assume that $V = V_1$ and $f(V_i) = V_i$ for each $i \in \{1, \dots ,k\}$ by replacing $f$ with suitable iterates. Assume that the conclusion to Theorem \ref{thm: mainfull} holds for each $H_j$. Then there exist non-negative integers $s_j$ such that each $x \in (\P^1_K)^n$ satisfies $f^s(x) \in H_j(K)$ for some non-negative integer $s$, and we have $f^{s_j}(x) \in H_j(K)$ for each $D+1 \leq j \leq n$.

    Now for each $x \in (\P^1_K)^n$ such that $f^s(x) \in V(K)$ for some non-negative integer $s$, we have certainly $f^s(x) \in H_j(K)$ for each $D+ 1 \leq j \leq n$. Thus letting $s'_0$ denote the quantity $\max_{D+ 1 \leq j \leq n}\{s_j\}$, we have $f^{s'_0}(x) \in \bigcap^n_{j = D+1} H_j(K)$ and $f^{s'_0 + 1}(x) \in H$. If $f^{s'_0 + 1}(x) \notin V$ then $f^{s'_0 + 1}(x) \in V_i$  for some $i \in \{2, \dots ,k\}$ such that $f^s(x) \in V \cap V_i$. Since $f(V) = V$ and $f(V_i) = V_i$, we have $f(V\cap V_i) \subseteq (V \cap V_i)$ and there exists $V' \subseteq V \cap V_i$ such that $f(V \cap V_i) = V'$ and $f(V') = V'$ after further replacing $f$ by a suitable iterate. Notice that $\dim(V') < D$ and $x$ is in the preimages of $V'$ under $f$. By the induction hypothesis, there exists a positive integer $s'_1$ such that $ (f^{-s'_1-1}(V') \setminus f^{-s'_1}(V))(K) = \emptyset$. Therefore, taking $s_0 = \max\{s'_1, s'_0 + 1\}$, we have $f^{s_0}(x) \in V(K)$. We proved that it is enough to prove that the conclusion to Theorem \ref{thm: mainfull} holds for each $H_j$.
    
    Now for each $H_j$, it is equivalent to prove the statement of Theorem \ref{thm: mainfull} for $\pi_j(V) \subset (\P^1)^{D + 1}$ which is an irreducible hypersurface invariant under $f' = (f_1, \dots ,f_D)$ and projecting dominantly onto any subset of $D$ coordinate axes.
\end{proof}
\begin{prop}\label{prop: cancellation=curve}
Let $K$ be a finitely generated field extension of $\Q$, $D$ a positive integer and let $L$ be the union of finite extensions of $K$ of degree less or equal to $D$. Let $f : \P^1_K \to \P^1_K$ be a surjective morphism of degree greater than $1$ defined over $K$. Then there exists a positive integer $N$ such that if $a ,b \in \P^1_L$ satisfies $f^n(a) = f^n(b)$ for some $n \geq 0$, then $f^N(a) = f^N(b)$.
\end{prop}
\begin{proof}
    \cite[Theorem 3.1]{DynamicalC} proved this for $L$ a number field. But the proof still works if we have $L$ as above. The only changes to the proof that are needed is that instead of embedding $L$ into $\Q_p$ for a suitable prime $p$ such that, after embedding, $f \in \Z_p[[x]]$, we embed $K$ into $\Q_p$ for a suitable prime $p$ such that $f \in \Z_p[[x]]$. Then $L$ is naturally embedded in the union of finite field extensions of $\Q_p$ of degree less or equal to $D$ which is inside $\C_p$. Notice that the key ingredient of the proof of \cite[Theorem 3.1]{DynamicalC} is the $p$-adic uniformization which works over $\C_p$. Also, in the proof \cite[Theorem 3.1]{DynamicalC}, $N$ is based on the least common multiple of the orders of roots of unity in a finite extension of $\Q_p$ which is a finite set. And Now, we only need to take the least common multiple of orders of roots of unity in $\overline{\Q}_p$ of degrees bounded by a constant depending only on $D$, $f$ and $K$. Notice that the set of roots of unity in $\overline{\Q}_p$ such that their degrees are all bounded by a constant is also finite \cite[Proposition 3.6(3)]{DynamicalC}. Thus, the proof follows just as in \cite[Theorem 3.1]{DynamicalC}.
\end{proof}
\begin{proof}[Proof of Theorem \ref{thm: mainfull} in the case when ${\rm deg}(f_1),\ldots ,{\rm deg}(f_n)\ge 2$.]
    By Lemma \ref{lem: red-to-surface}, it is enough to prove the statement with the assumption that $V$ is an irreducible hypersurface projecting dominantly onto any subset of $(n-1)$ coordinate axes in $(\P^1)^n$. Since $f(V) = V$, we have that $V$ contains a Zariski dense set of preperiodic points of $f$ \cite[Theorem 5.1]{invariant-dense-prep}. Therefore, by \cite[Theorem 2.2]{invariant-varieties-splitting}, we have if $n> 2$ then 
    \begin{enumerate}
        \item either $f_1, \dots , \dots f_n$ are all Latt\`es maps;
        \item or $f_i$'s are all conjugated to $x^{\pm d_i}$ or $\pm T_{d_i}$, where $T_{d_i}$'s are Chebyshev polynomials of degree $d_i = \deg(f_i)$. 
    \end{enumerate}
    Notice that by Lemma \ref{lem: equal-degree} in both cases $d_1 = d_2 = \dots = d_n$. If $n  = 2$, then by the proof of \cite[Theorem 1.3]{invariant-varieties-splitting2} we have either $f_1$, $f_2$ are both Latt\`es maps or neither of them is. So overall we have three separate cases to show: 
    \begin{enumerate}
        \item $n = 2$, $f_1$ and $f_2$ are not Latt\`es maps;
        \item $n > 2$, $f_i$'s are conjugated to either $x^{\pm d}$ or $\pm T_d$ for some $d  > 1$;
        \item $f_i$'s are all Latt\`es maps.
    \end{enumerate}

    Case (2) is implied by Proposition \ref{prop : toris} and Case (3) is implied by Proposition \ref{prop: Lattes}. We left to prove case (1): By \cite[Corollary 4.5]{Invariant-curve-rational} and \cite[Remark 4.3]{Invariant-curve-rational}, we have that there exists rational functions $U_1$, $U_2$ and $F$ defined over $\overline{K}$ such that \[U_1 \circ F = f_1 \circ U_1,\]
    \[U_2 \circ F = f_2 \circ U_2\]
    and $W_1 = (U_1, U_2)^{-1}(V)$ is a subvariety such that $(F,F)(W_1) \subseteq W_1$ and it contains an irreducible component $V' \subseteq W$ such that $(F,F)(V') = V'$ and $(F, F) (W_1) \subseteq V'$ after replacing $F$, $f_1$ and $f_2$ by some suitable iterate. Also by replacing $K$ with some larger number field, we assume $U_1$, $U_2$ and $F$ are defined over $K$. Let $K'$ be the union of finite field extension of $K$ whose degree is bounded by $\deg(U_1)^2 = \deg(U_2)^2$ of $K$. It is enough to show that there exists a positive integer $s_0$ such that 
    \[ ((F,F)^{-s-1}(V')\setminus (F,F)^{-s}(V'))(K') = \emptyset\]
    for any integer $s \geq s_0$. 
    
    Now if $F$ is Latt\`es map or is conjugate to either a power map or a Chebyshev polynomial or its negative, we obtain the result from Proposition \ref{prop : toris} and Proposition \ref{prop: Lattes} separately. So, we assume $F$ is not a Latt\`es map nor conjugate to either a power map or a Chebyshev polynomial or its negative. Then by \cite[Theorem 4.15]{Invariant-curve-rational}, we have there exists rational functions $U_3$, $U_4$, $F_1$ and $F_2$ over $\C$ such that 
    \[ U_3 \circ F_1 = F \circ U_3,\]
    \[ U_4 \circ F_2 = F \circ U_4\]
    and $W = (U_3, U_4)^{-1}(W_1)$ is a subvariety such that $(F_1, F_2)(W) \subseteq W$ and it contains a subvariety $W'$ such that $(F_1, F_2)(W') = W'$ after replacing $F_1$, $F_2$, $f_1$, $f_2$ and $F$ by a suitable iterate. Furthermore, $F_1$ and $F_2$ are not generalized Latt\`es maps by the Theorem. Now let $\tau_1=  U_1 \circ U_3$ and $\tau_2 = U_2 \circ U_4$ and replace $K$ by a finite generated field extension of $K$ such that $\tau_1$, $\tau_2$, $F_1$ and $F_2$ are all defined over $K$. After replacing $F$ and $F_1$, $F_2$ with a suitable iterate and abusing notation to let $W$ be some irreducible component of $W'$, we may assume $(F_1,F_2)(W) = W$. Let $L$ be the union of finite field extensions of $K$ of degree bounded by $\deg(\tau_1)^2 = \deg(\tau_2)^2$. Similarly, it is enough to show that there exists a positive integer $s_0$ such that 
    \[ ((F_1,F_2)^{-s-1}(W)\setminus (F_1,F_2)^{-s}(W))(L) = \emptyset\]
    for all integers $s \geq s_0$.

   In this case, we use \cite[Theorem 1.1]{Invariant-curve-rational} and we get that there exists rational functions $X_1, X_2, Y_1, Y_2$ and $B$ such that there exists some positive integer $d$ satisfies 
    \begin{equation}\label{eq: decompF_1}
    F^d_1 = X_1 \circ Y_1
    \end{equation}
    \begin{equation}\label{eq: decompF_2}
     F^d_2 = X_2 \circ Y_2
    \end{equation}
    \begin{equation}\label{eq: decompF_3}
        B^d = Y_1 \circ X_1 = Y_2 \circ X_2, 
    \end{equation}
  $W \subseteq (Y_1,Y_2)^{-1}(\Delta)$ and also $(F^d_1, F^d_2)((Y_1, Y_2)^{-1}(\Delta)) = W$.
    
    Enlarging $K$ by adjoining coefficients of $X_1, Y_1, X_2$ and $ Y_2$ and abuse the notation to let $K$ denote this larger field and $L$ be the union of finite field extension of degree bounded by $\deg^2(\tau_1) = \deg^2(\tau_2)$ of K.
   
    Now it is enough to show that there exists a non-negative integer $s_0$ such that for any $s \geq  s_0$,
    \begin{equation}
        ((F_1, F_2)^{(-s-1)d}(W) \setminus (F_1, F_2)^{-sd}(W))(L) = \emptyset.
    \end{equation}
    We claim that to prove the above, it suffices to show that there exists $s_0$ such that for any $s \geq s_0$
   \begin{equation}\label{eq: largervariety}
       ((F_1, F_2)^{d(-s-1)} ((Y_1,Y_2)^{-1}(\Delta)) \setminus (F_1, F_2)^{-ds} ((Y_1,Y_2)^{-1}(\Delta)))(L) = \emptyset.
   \end{equation}
   To see this, suppose there exists $s'_0$ such that for any $s \geq s'_0$ Equation (\ref{eq: largervariety}) holds. If there exists a non-negative integer $s > s'_0$ and a $x \in \P^1_L \times \P^1_L$ such that $(F_1, F_2)^{sd}(x) \in W(L)$ but $(F_1,F_2)^{s'_0d}(x) \notin W$, then $(F_1,F_2)^{s'_0d}(x) \in (Y_1, Y_2)^{-1}(\Delta) \setminus W$ by Equation (\ref{eq: largervariety}). Thus $(F_1,F_2)^{(s'_0 + 1)d}(x) \in W$. Therefore, take $s_0 = s'_0 + 1$, we have for any $s \geq s_0$
   $$((F_1, F_2)^{(-s-1)d}(W) \setminus (F_1, F_2)^{-sd}(W))(L) = \emptyset. $$
   Thus, it is enough to show that there exists $s_0$ such that for any $s \geq s_0$
   \begin{equation}
       ((F_1, F_2)^{d(-s-1)} ((Y_1,Y_2)^{-1}(\Delta)) \setminus (F_1, F_2)^{-ds} ((Y_1,Y_2)^{-1}(\Delta)))(L) = \emptyset
   \end{equation}
   which is equivalent to, by Equation (\ref{eq: decompF_1}), (\ref{eq: decompF_2}) and (\ref{eq: decompF_3}), 
    \begin{equation}
        (Y_1, Y_2)^{-1}((B^d, B^d)^{-s-1}(\Delta)\setminus (B^d, B^d)^{-s}(\Delta))(L) = \emptyset.
    \end{equation}
    Since, again, $Y_1, Y_2$ are defined over $L$, it is enough to show 
    \begin{equation}\label{eq: BBpreimages}
        ((B^d, B^d)^{-s-1}(\Delta)\setminus (B^d, B^d)^{-s}(\Delta))(L) = \emptyset.
    \end{equation}
    
    Notice that if there exists a non-negative integer $s_0$ such that for any non-negative integer $s $ and any $x , y \in \P^1_L$, we have 
    \begin{equation}
        B^{ds}(x) = B^{ds}(y)
    \end{equation}
    implies 
    \begin{equation}
        B^{ds_0}(x) = B^{ds_0}(y),
    \end{equation}
    then Equation (\ref{eq: BBpreimages}) holds for $s \geq s_0$. While this is proved by Proposition \ref{prop: cancellation=curve}. The result follows.

\end{proof}
\begin{proof}[Proof of Theorem \ref{thm: mainfull} in the general case.]
By replacing $f$ by some suitable iterate, we have that each irreducible component of $V$ is also an invariant subvariety of $f$. It is enough to prove the theorem for each irreducible component of $V$, so we assume that $V$ is irreducible. 

If there doesn't exist $f_i$ such that $\deg(f_i) = 1$, for $i \in \{1, \dots, n\}$, then the statement has been proved.  Thus we may assume that at least one $f_i$ is an automorphism. We reorder the coordinates so that there exists a positive integer $k \in \{2, \dots , n-1\}$ such that $\deg(f_i) > 1$ when $i \leq k$ and $\deg(f_i) = 1$ when $i > k$. Now Lemma \ref{lem:reduce-to-alldegreenontrivial} implies that $V = V_1 \times V_2$, where $V_1\subseteq (\P^1)^k$, $V_2 \subseteq (\P^1)^{n-k}$ such that $g_1(V_1) = V_1$ and $g_2(V_2) = V_2$, where $g_1 = (f_1, \dots, f_k )$ and $g_2 = (f_{k+1}, \dots, f_n)$. Notice that $g_2^{-1}(V_2) = V_2$ and there exists a non-negative integer $s_0$ such that $(g_1^{-s-1}(V_1) \setminus g_1^{-s}(V_1)(K) = \emptyset$ for all $s\geq s_0$, since we have established Theorem \ref{thm: mainfull} in the case when the maps all have degree $\ge 2$. Thus, or any non-negative integer $s \geq s_0$,
\[ (f^{-s-1}(V)\setminus f^{-s}(V))(K) \]
\[= (g_1^{-s-1}(V_1) \setminus g_1^{-s}(V_1))(K) \times (g_2^{-s-1}(V_2) \setminus g_2^{-s}(V_2))(K) = \emptyset.\]
The result follows.

\end{proof}
\section*{Acknowledgments} The author thanks Jason Bell for helpful comments and suggestions and Tom Tucker for inspiring discussions. The author thanks Yohsuke Matzusawa and Kaoru Sano for their helpful comments on an earlier version of the draft and for sharing their recent preprint \cite{OnPreimagesProblem} where they proved the same result with $n = 2$ but used a different approach. The author also thanks Junyi Xie for pointing out the reference \cite{Xie2023}.
\bibliographystyle{amsplain}
\bibliography{XiaopaperNov62023}
\end{document}